\documentclass[12pt]{amsart}
\usepackage{amsmath,amssymb}
\usepackage{eufrak}
\usepackage{amsthm}
\usepackage{enumitem}
\usepackage{graphicx}
\usepackage{calc}
\usepackage{graphics}
\usepackage{color}
\usepackage[dvipsnames]{xcolor}
\usepackage{hyperref}
\usepackage{youngtab}
\usepackage{marginnote}
\usepackage{lscape}
\usepackage{tikz}
\usetikzlibrary{calc,shadings,patterns}
\usetikzlibrary{matrix}
\usepackage{faktor}
\usepackage{hyperref}
\usepackage[normalem]{ulem}
\usepackage{bm}
\usepackage{afterpage}
\usepackage{mathrsfs}
\usepackage[all]{xy}
\usetikzlibrary{arrows}
\usetikzlibrary{matrix}
\usepackage{verbatim}
\usepackage{setspace}

\usepackage{multirow}
\usetikzlibrary{decorations.pathreplacing}
\usepackage{enumitem,pifont,xcolor}
\definecolor{myblue}{RGB}{0,29,119}

\newtheorem{theorem}{Theorem}[section]
\newtheorem{proposition}[theorem]{Proposition}
\newtheorem{corollary}[theorem]{Corollary}
\newtheorem{lemma}[theorem]{Lemma}
\theoremstyle{definition}
\newtheorem{definition}[theorem]{Definition}

\newtheorem{remark}[theorem]{Remark}

\newtheorem{problem}[theorem]{Problem}

\newtheorem*{theorem*}{Theorem}

\makeatletter
\@addtoreset{equation}{section}
\makeatother

\newcommand{\Zz}{{\mathbb Z}} 
\newcommand{\Rr}{{\mathbb R}} 

\newcommand{\cA}{{\mathcal A}}

\newcommand{\cL}{{\mathcal L}}

\newcommand{\cR}{{\mathcal R}}

\DeclareMathOperator{\Hom}{Hom}
\DeclareMathOperator{\ext}{Ext}

\providecommand{\AMS}{$\mathcal{A}$\kern-.1667em%
\lower.25em\hbox{$\mathcal{M}$}\kern-.125em$\mathcal{S}$}

\hypersetup{ colorlinks,
 linkcolor={blue},  citecolor={BrickRed}, urlcolor={blue}}

\begin{document}

\let\thefootnote\relax\footnotetext{
Contact: \href{mailto:emre-sen@uiowa.edu}{emre-sen@uiowa.edu} }

\title[Weak Exceptional Sequences ]{Weak exceptional sequences }
\begin{abstract}
We introduce weak exceptional sequence of modules which can be viewed as another modification of the standard case, different than the works of Igusa-Todorov \cite{Igusa-Todorov} and Buan-Marsh \cite{Buan-Marsh}. For hereditary algebras it is equivalent to standard exceptional sequences. One important new feature is: if global dimension of algebra is greater than one, then the size of the full sequence can exceed the rank of the algebra. We use both cyclic and linear Nakayama algebras to test combinatorial aspects of this new sequence. For some particular classes we give closed form formulas which returns the number of the full weak exceptional sequences, and compare them with the number of exceptional sequences of types $\mathbb{A}$ and linear radical square zero Nakayama algebras \cite{sen19}.
\end{abstract}
 
\author{Emre SEN}
\maketitle
\tableofcontents
\section{Introduction}
An indecomposable representation $M$ of quiver $Q=(Q_0,Q_1)$ is called exceptional if $\Hom_{kQ}(M,M)\cong k$ and $\ext^1_{kQ}(M,M)=0$ where $kQ$ is path algebra. Sequence of representations $(M_1,\ldots,M_r)$ is called exceptional if every representation $M_i$, $M_j$   in the sequence satisfies  $\Hom_{kQ}(M_j,M_i)=0$ and $\ext^1_{kQ}(M_j,M_i)=0$ where $1\leq i<j\leq r$. Under this set up, Crawley-Boevey \cite{craw} and Ringel \cite{ringel1994braid} showed that there is an action of braid group on the sequences of exceptional modules. The number of exceptional sequences for Dynkin types were calculated by Seidel \cite{seidel} and their combinatorial descriptions from different aspects were studied in the papers \cite{ringel}, \cite{igusa}. 
In the literature, there are different directions involving exceptional sequences:
\begin{enumerate}[label=\roman*)]
\item \label{standar} Instead of path algebra, one can study exceptional modules and exceptional sequences over any finite dimensional algebra $A$ by setting all higher extensions to zero i.e. module $M$ is said to be exceptional if $\Hom_A(M,M)=0$ and $\ext^{r}_A(M,M)=0$ for all $r$. Indeed this is the standard definition of exceptional module and it was introduced by Rudakov for vector bundles, and studied by his school \cite{rudakov}.

\item Exceptional objects can be defined over derived categories, definition is extended simply  by considering shifted modules. For example, the number of exceptional sequences of derived $\mathbb{A}$ type is studied in \cite{araya} and related to the Catalan combinatorics.
\item In \cite{Igusa-Todorov}, Igusa and Todorov introduced signed exceptional sequences for hereditary algebras and proved that there is a bijection between the set of ordered cluster tilting sets and the set of complete signed exceptional sequences. Later, Buan and Marsh introduced $\tau$-exceptional sequences \cite{Buan-Marsh} which works for any finite dimensional algebra and gives the same result of \cite{Igusa-Todorov} if algebra is hereditary.

\end{enumerate}

In this work we consider another direction:  we will modify actual definition of an exceptional module and exceptional sequence in the following manner: 
\begin{definition} Let $A$ be a finite dimensional algebra.
\begin{itemize}
\item A module $M$ in mod-$A$ is called \emph{weak exceptional} if  
\begin{align}\label{def1}
\Hom_A(M,M)\cong k\hspace{1cm}\text{and}\hspace{1cm} \ext^1_A(M,M)=0
\end{align}
\item We call a pair of modules $(M,N)$ exceptional if both $M,N$ are weak exceptional together with :
\begin{align}\label{def2}
\Hom_A(N,M)=0,\hspace{2cm} \ext^1_A(N,M)=0
\end{align}
i.e. $M$ is right perpendicular to $N$ as in \cite{lenzig} and \cite{ringel1994braid}.
\item A k-tuple of modules $(M_1,M_2,\ldots,M_k)$ is called \emph{weak exceptional collection} or \emph{weak exceptional sequence} if $(M_i,M_j)$ is a weak exceptional pair for all $1\leq i<j\leq k$.
\item By \emph{size} of an exceptional collection, we mean the number of distinct modules in it. It is called \emph{full}  if size of the given exceptional collection is equal to:
\begin{align} \label{size}
\sup\left\{size(E)\,\vert\,\, E\,\, \text{is weak exceptional sequence of}\,A\right\}
\end{align}

\end{itemize}
\end{definition}

 The need is because of the following: in the case of hereditary algebras, higher extension are already trivial and study of exceptional sequences is closely related to study of perpendicular module categories \cite{craw}, \cite{ringel1994braid}. However, even in the case of self injective algebras, modules have infinite projective dimensions, therefore they have nontrivial higher self extensions. Moreover, there might not be an exceptional module in the sense of \ref{standar} by proposition \ref{weakvs}.

\begin{remark}\begin{enumerate}[label=\arabic*)]\item By exceptional sequence of modules, we mean exceptional sequences of isomorphism classes of modules.
\item Notice that we allow higher extensions in definitions \ref{def1}, \ref{def2} i.e. $\ext^i_A(M,N)$, $i\geq 2$ can be nontrivial. Right and left perpendicular pairs of modules over finite dimensional algebras were considered in \cite{perp} which uses the similar conditions. Here, instead of pairs, we consider the full sequences.
\item In the hereditary case, size of a full collection is equal to the number of simple modules. However, we will construct weak exceptional collections such that size of the full collection will be greater than the number of the simple modules. Therefore, we pose the following question:
\begin{problem}\label{question} What is the size of the full weak exceptional collection for algebras of finite representation type excluding hereditary case? And then, what is the number of the full weak exceptional sequences?
\end{problem}
\end{enumerate}
\end{remark}

We give partial answer to the question. Before stating the results, we use the following notation:
$\Lambda^k_n$ denotes selfinjective Nakayama algebra over $n\geq 2$ vertices and length of each indecomposable projective is $k$, where $k\geq 2$. Orientation is given by $\eta_i:i\mapsto i+1$ for $1\leq i\leq n-1$ and $\eta_n:n\mapsto 1$ where $1\leq i\leq n$ are vertices. 
We use the notation:
\begin{gather*}
 s\left(\Lambda^k_n\right)=\sup\left\{size(E)\,\vert\,\,E\,\, \text{is weak exceptional collection of}\,\,\Lambda^k_n\right\}\\
 \#(\Lambda^k_n)=\left|\left\{E\mid E\,\, \text{is the full weak exceptional collection of}\,\,\Lambda^k_n \right\}\right|
\end{gather*}
i.e. it is the number of distinct full weak exceptional collections.

Here we list main results of the paper:
\begin{theorem}\label{thm1}
For the cyclic Nakayama algebra $\Lambda^n_n$, the size of the full collection $s\left(\Lambda^n_n\right)$ is $n$, and the number of distinct full exceptional sequences $\#(\Lambda^n_n)$ is $n^n$.
\end{theorem}
\begin{theorem}\label{thm2}
For the cyclic Nakayama algebra $\Lambda^{n-1}_n$, the size of the full collection $s\left(\Lambda^{n-1}_n\right)$ is $2n-2$, and the number of distinct full exceptional sequences $\#(\Lambda^{n-1}_n)$ is $n$.
\end{theorem}

\begin{theorem}\label{thm3}
For the cyclic Nakayama algebras $\Lambda^2_n$, we have:\\
\begin{center}
$s(\Lambda^2_n)=\begin{cases}
3k+1& \text{if}\quad n=2k+1\\
3k-1& \text{if}\quad n=2k\\
\end{cases}
$\\\vspace{0.5cm}
$\#(\Lambda^2_n)=\begin{cases}
2k+1& \text{if}\quad n=2k+1\\
2k\left(\cfrac{8^k}{12}-\cfrac{(-1)^k}{3}+1\right)& \text{if}\quad n=2k\\
\end{cases}
$
\end{center}
\end{theorem}

In the next section, we characterize Hom and Ext orthogonal modules in terms of regions in the Auslander-Reiten quiver. We give proofs of theorems \ref{thm1}, \ref{thm2}, \ref{thm3} in the latter sections \ref{ch3}, \ref{ch4} and \ref{ch5} respectively. We examine further properties of (weak) exceptional sequences in the last section \ref{last}. I am thankful to S. Zhu for many useful discussions, and to G. Todorov for stimulating interest in this work.

\section{Lattices}
In this work, to study exceptional sequences, it is convenient for us to identify Auslander-Reiten quiver of cyclic Nakayama algebras by a lattice spanned by vectors $f_1=(2,0)$ and $f_2=(1,1)$ in $\Rr^2$. 
The lattice 
\begin{align}\label{arlattice}
\cA\cR\left(\Lambda^k_n\right)=\left\{af_1+bf_2\,\vert\,\, 0\leq b\leq k-1,\,\text{and}\,a,b\in\Zz \right\}
\end{align}
can be identified with the AR quiver of $\Lambda^k_n$ and 
its fundamental domain is:
\begin{align}\label{fundamentallattice}
\cL\left(\Lambda^k_n\right)=\left\{af_1+bf_2\,\vert\,\, 0\leq b\leq k-1,\,\, 0\leq a\leq n-1,\,\text{and},\,a,b\in\Zz \right\}
\end{align}
We can identify the integral points of $\cL\left(\Lambda^k_n\right)$ by indecomposable $\Lambda^k_n$ modules via the map:
\begin{gather*}
\Gamma: \text{mod}\,\Lambda^k_n\rightarrow \cL\left(\Lambda^k_n\right)\\
M\rightarrow\left(n-\text{isoc}(M),l(M)-1\right)
\end{gather*}

where $\text{isoc}(M)$ is the index of the socle of $M$ and $l(M)$ is the length of the module $M$.  $\Gamma^{-1}$ is well defined on the fundamental domain $\cL\left(\Lambda^k_n\right)$ and injective, explicitly: the tuple $(a,b)$ gives the indecomposable module whose length is $b+1$ and index of the socle is $n-a$. Therefore, we can extend its domain in the following way:  for any point $(a,b)\in\cA\cR(\Lambda^k_n)$, we can define the inverse map $\Gamma^{-1}(a,b):=\Gamma^{-1}((a)_n,b)$ where 
$(a)_n$ is $a\!\!\mod n$.

To exploit symmetric structure of the fundamental domain, we define translation map $\sigma$ of module $M$:
\begin{align*}
\sigma(M)=\Gamma^{-1}\left(\Gamma(M)_1+1,\Gamma(M)_2\right)
\end{align*}
Notice that, for any nonprojective module $M$, $\sigma(M)$ is inverse of Auslander-Reiten translate $\tau^{-1}(M)$. If $M$ is projective module at vertex $i$, then $\sigma(P_i)\cong P_{i+1}$.

\begin{remark} Because $\Lambda^k_n$ is  self injective, for every indecomposable module $M$, there is a unique projective module $P$ such that $M\hookrightarrow P$ i.e. $M$ is submodule of $P$. Indeed $P$ is injective envelope of $M$.
\end{remark}

\begin{definition}\label{region} Let $p_1,p_2,p_3$ and $p_4$ be integral points of the lattice $\cA\cR(\Lambda^k_n)$. We denote the trapezoidal region bounded by the lines $\overline{p_1p_2}$, $\overline{p_2p_3}$, $\overline{p_3p_4}$ and $\overline{p_4p_1}$ by $\overline{p_1p_2p_3p_4}$.
\end{definition}
It is obvious that shape of regions in $\cA\cR(\Lambda^k_n)$ are translation $\sigma$-invariants. 

\begin{proposition}\label{homextortho} Let $M$ be $\Lambda^k_n$ module $k\leq n$, where $\Gamma(M)=(a,b)$. Then:
\begin{align*}
\Hom(M,X)\neq 0 &\iff X=\Gamma^{-1}(p),\, p\in \overline{(a,b)(a+b,0)(a+b,k-1)(a,k-1)} \\
\ext(M,X)\neq 0 &\iff X=\Gamma^{-1}(p),\, p\in \overline{(a\!-\!1,b)(a\!-\!1,0)(a\!-\!k\!+\!1,k\!-\!2)(a\!+\!b\!-\!k\!+\!1,k\!-\!2)} 
\end{align*}

\end{proposition}

\begin{proof} Let $M$ be a $\Lambda^k_n$ module and $top(M)$, $I(M)$ be top of $M$ and injective envelope of $M$. The morphisms  $\Hom(M,I(M))$, $\Hom(M,I(top(M)))$ are nonzero.
We can construct the trapezoidal region $T$  by using coordinates in the lattice $(a,b),(a+b,0),(a,k-1)$ and $(a+b,k-1)$. We made that choice for coordinates since if $\Gamma(M)=(a,b)$ then $\Gamma(top(M))=((a+b)_n,0)$, $\Gamma(I(M))=(a,k-1)$ and $\Gamma(I(top(M)))=((a+b)_n,k-1)$.\\
Let $Q^i(M)$ be the quotient module of $M$ by radical power $i$ i.e. $\faktor{M}{rad^i(M)}$. The map  $M\rightarrow I(soc(Q^i(M)))$ is nonzero and it factors through a module $X$ where $X=\Gamma^{-1}(p)$ and $p\in T$. Therefore $\Hom(M,\Gamma^{-1}(p))\neq 0$ if and only if $p\in T$. By using similar arguments, the second statement follows.
\end{proof}

\begin{proposition}  $(M_1,M_2)$ is weak exceptional pair if and only if $(\sigma(M_1),\sigma(M_2))$ is exceptional pair.
\end{proposition}

\begin{proof} By proposition \ref{homextortho}, we have $\Hom(M,N)=0\iff\Hom(\sigma(M),\sigma(N))=0$ and $\ext(M,N)=0\iff\ext(\sigma(M),\sigma(N))=0$, therefore exceptional pairs are preserved under translation $\sigma$.\end{proof}

\begin{corollary}$(M_1,M_2,\ldots,M_j)$ is weak exceptional sequence if and only if \\$(\sigma(M_1),\sigma(M_2),\ldots,\sigma(M_j))$ is weak exceptional sequence.
\end{corollary}

\begin{proposition}\label{atype} Consider the region $\left\{(x,y)\in\cL(\Lambda^k_n)\,\,\vert\,\, 0\leq x,y, \text{and}\,\,x+y\leq s\right\}$. If $x<a$ and $(0,s)$ is in the region, then it is identified by Auslander-Reiten quiver of $\mathbb{A}_s$.
\end{proposition}

\begin{proof}
The integral points of the region can be identified by the vertices Auslander-Reiten quiver of $\mathbb{A}_s$. Moreover restriction of the maps in $\cL$ to it, makes it Auslander-Reiten quiver of type $\mathbb{A}$.
\end{proof}

Since translation $\sigma$ preserves the shapes, we treat translations of regions in proposition \ref{atype} as AR quiver of type $\mathbb{A}$.

\section{{Study of $\#(\Lambda^n_n)$}}\label{ch3}

In this section, we study full weak exceptional sequences of $\Lambda^n_n$. We start with an obvious observation:

\begin{lemma} A full weak exceptional collection of $\Lambda^n_n$ cannot contain more than one projective module.
\end{lemma}

\begin{proof}
If $P_i$, $P_j$ are projective modules of $\Lambda^n_n$, $\Hom_{\Lambda^n_n}(P_i,P_j)$ and $\Hom_{\Lambda^n_n}(P_j,P_i)$ exist, therefore there can be at most one projective in the full collection. 
\end{proof}

It is useful, since it guarantees number of projective modules in the full collection is less than or equal to one. The key observation, in order to study the number of full weak exceptional sequences of $\Lambda^n_n$ is the following:

\begin{proposition}\label{keyprop1}
The size of full weak exceptional sequence of $\Lambda^n_n$ is $n$ i.e. $s\left(\Lambda^n_n\right)=n$, and any full collection contains exactly one projective $\Lambda^n_n$ module.
\end{proposition}

\begin{proof}
Notice that the collection $\left\{(0,b), b=0,1,\ldots n-1\right\}$ in $\cL\left(\Lambda^n_n\right)$ forms a weak exceptional sequence in $\Lambda^n_n$, therefore $s\left(\Lambda^n_n\right)\geq n$.\\
Assume that there is a projective module in the collection. By using $\sigma$ action, we can choose it as $P_2$. Hom orthogonal modules can be interpreted as the following regions in the lattice $\cL(\Lambda^n_n)$ \ref{fundamentallattice} by proposition \ref{homextortho}
\begin{align*}
\Hom(X,P_2)=0 &\iff \Gamma(X)_1+\Gamma(X)_2\leq n-2\\
\Hom(P_2,X)=0 &\iff \begin{cases}
\Gamma(X)_1+\Gamma(X)_2\leq n-3,\quad\text{or}\\
\Gamma(X)_1=n-1\quad \text{and}\quad \Gamma(X)_2=0\ldots n-2\\
\end{cases}
\end{align*}
Two regions are equivalent with respect to $\sigma$ action, moreover can be identified by Auslander-Reiten triangle of $\mathbb{A}_{n-1}$ type quiver by proposition \ref{atype}. It is well known that size of exceptional collections over $\mathbb{A}_{n-1}$ is $n-1$ \cite{ringel1994braid}, therefore we get the size $n$ for weak exceptional collections containing a projective.\\
Now we show that if there is no projective module in the collection, its size cannot be $n$. We prove it by induction. Assume that there is a radical of a projective in a weak exceptional collection. Proposition \ref{homextortho} gives the conditions:
\begin{align*}
\Hom(X,rad(P_2))=0\iff \Gamma(X)_1+\Gamma(X_2)\leq n-2\\
\ext(X,rad(P_2))=0 \iff \Gamma(X)_1+\Gamma(X_2)\leq n-3\nonumber
\end{align*}
Notice that we already excluded $P_2$ because of our assumption. Now the region we obtained is of type $\mathbb{A}_{n-2}$. Therefore the size of weak exceptional collection is smaller or equal than $n-1$. With respect to $\sigma$ action, we get the same type quiver for the set $\left\{X\vert\Hom(rad(P_2),X)=0\,\,\text{or}\,\ext(rad(P_2),X)=0\right\}$. Therefore, if there is no projective module in a collection, none of the radicals can be in it. Similar conditions can be obtained for radical squares of projective modules in the fundamental domain where $b\leq n-3$. The region is disjoint union of two $\mathbb{A}$ type quivers. By induction on the radical powers, statement follows.
\end{proof}

\begin{corollary} Any full weak exceptional collection of $\Lambda^n_n$ contains exactly one projective module.
\end{corollary}

Now, we can solve the enumerative problem for $\Lambda^n_n$:
\begin{proposition}\label{structureofseq}
Any full weak exceptional sequence of $\Lambda^n_n$ can be obtained by adding one projective module to an exceptional collection of $\mathbb{A}_{n-1}$ of size $n-1$.
\end{proposition}
\begin{proof}
Let $\left(X_1,\ldots,X_{i-1},X_i,\ldots,X_{n-1}\right)$ be an exceptional collection of $\mathbb{A}_{n-1}$. By proposition \ref{atype} and $\sigma$-action, we can obtain its AR quiver from $\Lambda^n_n$. Hence all of the modules in the sequence have projective covers and injective envelopes in the larger algebra. We will show existence of projective $\Lambda^n_n$ module $Y$, to add between $X_{i-1}$ and $X_{i}$ in the above sequence. Observe that $Y$
 cannot be projective cover of $X_1,\ldots,X_{i-1}$ and injective envelope of $X_i,\ldots,X_{n-1}$ by definition of weak exceptional sequence. In the extreme case, the total number of projective modules that $Y$ cannot be isomorphic is at most $n-1$. Therefore, this guarantees a projective choice for $Y$.
 \end{proof}

\begin{theorem}[Thm \ref{thm1}] The number of the full weak exceptional sequences is:
$\#\left(\Lambda^n_n\right)=n^n$
\end{theorem}
\begin{proof}
We prove the statement by induction. \\
The number of exceptional collections of $\mathbb{A}_{n-1}$ is  $\#\left(\mathbb{A}_{n-1}\right)=n^{n-2}$ by the result of Seidel \cite{seidel}. By previous proposition \ref{structureofseq}, we can add a projective module and there are $n$ possible places, which gives $n. \#\left(\mathbb{A}_{n-1}\right)=n^{n-1}$. Because $\sigma$ acts transitively, i.e. for each distinct projective module there are $n^{n-1}$ exceptional collections, in total we get $n^n$.
\end{proof}

\section{Study of $\#\left(\Lambda^{n-1}_n\right)$}\label{ch4}

We study full weak exceptional sequences of $\Lambda^{n-1}_n$. To do this, we use some properties of weak exceptional sequences of $\Lambda^n_n$ and $\mathbb{A}_{n-1}$. First we find the size:

\begin{lemma} The size $s\left(\Lambda^{n-1}_n\right)$ can be at most $2n-2$.
\end{lemma}

\begin{proof}

Let $E$ be a full weak exceptional collection of $\Lambda^{n}_n$. By proposition \ref{keyprop1}, it contains exactly one projective module. By the action of $\sigma$, without loss of generality, we can choose that projective as $P_1$. Now, the remaining modules of $E$ are integral points $(a,b)$ of the fundamental domain $\cL\left(\Lambda^{n}_n\right)$ satisfying $a+b\leq n-1$.\\
Let $E'$ be the collection obtained by removing projective $P_1$ form $E$. Therefore, by proposition \ref{atype}, $E'$ is exceptional $\mathbb{A}_{n-1}$ sequence. Notice that an exceptional collection of $\Lambda^{n-1}_n$ can be obtained by extending exceptional $\mathbb{A}_{n-1}$ sequences with proper submodules of the removed projective module $P_1$. In $\cL\left(\Lambda^{n}_n\right)$, this is equivalent to choosing modules from the region such that $a+b\leq n-2$ and union by $a+b=n-1$ and $a\geq 1$.

Moreover, the size of $E'$ is $n-1$. Since the additional line in $\cL(\Lambda^{n-1}_n)$ i.e. $a+b=n-1$ with $a\geq 1$  can contribute at most $n-1$ terms, upper bound for the size of an exceptional sequence is $2n-2$. 
\end{proof}

We will construct a weak exceptional sequence of size $2n-2$. Notice that the collection $\left((1,n-2),\ldots,(n-1,0)\right)$ is weak exceptional in $\cL\left(\Lambda^{n-1}_n\right)$ and we denote it by $F$.

\begin{proposition}\label{keyprop2}
$\left(E,F\right)$ is weak exceptional collection of $\Lambda^{n-1}_n$ if and only if $E=\left((0,0),(0,1),\ldots,(0,n-1)\right)$. In other words, $E$ consists of all submodules of the $\Lambda^{n-1}_n$-projective module $P_2$.
\end{proposition}

\begin{proof}
It is easy to verify only if part.\\
Observe that $E$ cannot contain $(n-2,0)$ since $(n-1,0)$ extends it. This holds for any $(a,b)$, $a+b=n-2$ and $a\neq 0$, since module $(n-1,0)$ in $F$ extends them. We get similar pattern for the modules $(c,d)$, $1\leq c\leq n-2$, $0\leq d\leq n-3$ because, if $c+d=n-k$, then $(n-k-1,k-2)$ extends them nontrivially. The only remaining possibility is $E=((0,0),(0,1),\ldots,(0,n-1))$.
\end{proof}

\begin{corollary}
There is only one full weak exceptional sequence for linear Nakayama algebra whose Kupisch series \footnote{ i.e. ordered  lengths of indecomposable projective modules of linear Nakayama algebra}  are $\left(n-1,n-1,n-2,\ldots,2,1\right)$.
\end{corollary}

\begin{theorem} [Thm \ref{thm2}] The number of the full weak exceptional sequences is: $\#\left(\Lambda^{n-1}_n\right)=n$
\end{theorem}
\begin{proof}

We need to show that with respect to $\sigma$-action, there is only one full weak exceptional collection. But, this is a corollary of proposition \ref{keyprop2}. Because, any full weak exceptional collection is obtained by removing the unique projective module of $E$ in $\Lambda^n_n$ and then adding modules to it. The total number has to be $n$ by periodicity of $\sigma$.
\end{proof}

\section{Study of $\#\left(\Lambda^{2}_n\right)$}\label{ch5}

Here, we aim to prove the theorem \ref{thm3}. Parity of $n$ leads distinct results. Therefore we have two main cases which we analyze $n$ is odd (proposition \ref{nisodd}) and even (theorem \ref{niseven}) separately. For simplicity we use $[i]$ and $[i,i+1]$ to show simple module $S_i$ and projective module $P_i$ respectively. It is clear that $[n,1] $ is $P_n$.
\subsection{n is odd}
It is enough to show that there is only one full exceptional sequence with respect to $\sigma$ action. To visualize them we introduce 'bone' method: Assume that $n=3$, we can identify the weak exceptional sequence $\left([3],[2,3],[1,2],[1]\right)$ by the following figure:

\begin{center}
\begin{align}\label{bone}
\begin{tikzpicture}
\draw (0,0) node {$\bullet$}--(1,1) node {$\bullet$}--(2,1) node {$\bullet$} -- (3,0) node {$\bullet$};
\end{tikzpicture}
\end{align}

\end{center}

If we choose $n=5$, an obvious weak exceptional collection is $\left([5],[4,5],[3,4],[2,3],[1,2],[1]\right)$. However it is not full, we can add one more module $[3]$, and bone will be :

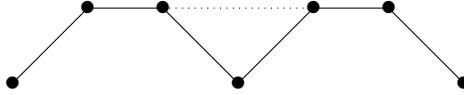
\begin{center}

\begin{tikzpicture}
\draw (0,0) node {$\bullet$}--(1,1) node {$\bullet$}--(2,1) node {$\bullet$} -- (3,0) node {$\bullet$}--(4,1) node {$\bullet$} -- (5,1) node {$\bullet$} -- (6,0) node {$\bullet$} ;
\draw[dotted] (2,1) -- (4,1);

\end{tikzpicture}
\end{center}

We give one more example, if $n=7$:

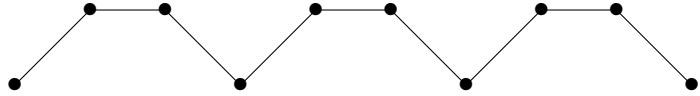
\begin{center}

\begin{tikzpicture}
\draw (0,0) node {$\bullet$}--(1,1) node {$\bullet$}--(2,1) node {$\bullet$} -- (3,0) node {$\bullet$}--(4,1) node {$\bullet$} -- (5,1) node {$\bullet$} -- (6,0) node {$\bullet$}-- (7,1)node {$\bullet$} --  (8,1)node {$\bullet$} --  (9,0)node {$\bullet$}  ;

\end{tikzpicture}
\end{center}

\begin{proposition}[Thm \ref{thm3} a]\label{nisodd} The size and the number of weak exceptional sequences of type $\Lambda^2_{2k+1}$ is given by $s\left(\Lambda^2_{2k+1}\right)=3k+1$,
$\#\left(\Lambda^2_{2k+1}\right)=2k+1$ 
\end{proposition}
\begin{proof}
We use induction. One can verify directly that $\left([3],[2,3],[1,2],[1]\right)$ is a full exceptional sequence for $\Lambda^2_3$.\\
Notice that a full weak exceptional collection $E=\left([n],[n-1,n],\ldots,[1,2],[1]\right)$ of $\Lambda^2_n$ is also a weak exceptional collection for $\Lambda^2_{n+2}$. The $\Lambda^2_{n+2}$ modules we can add to the collection $E$ are $[n+2],[n+1,n+2],[n+1],[n,n+1]$. It is impossible to add four of them. One can add at most three of them and the only option is $[n+2],[n+1,n+2],[n,n+1]$. Hence the claims follow. 
\end{proof}

\subsection{n is even}
This is drastically different than the previous case. First we give analysis of $n=4$ case in details and later we develop method to join 'bones'. There are $4$ full exceptional collections by direct computation and we list them below:
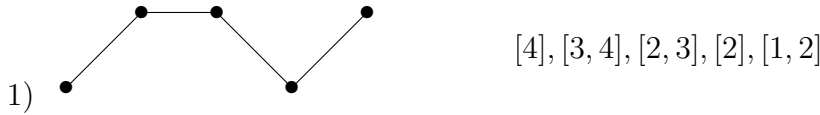
\begin{enumerate}[label=\arabic*)]
\item 
\begin{tikzpicture}
\draw (0,0) node {$\bullet$}--(1,1) node {$\bullet$}--(2,1) node {$\bullet$} -- (3,0) node {$\bullet$} --(4,1) node {$\bullet$}  (8,0.5) node {$[4],[3,4],[2,3],[2],[1,2]$};
\end{tikzpicture}
\item 
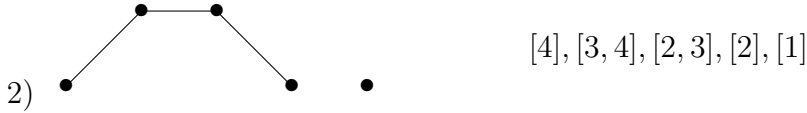
\begin{tikzpicture}
\draw (0,0) node {$\bullet$}--(1,1) node {$\bullet$}--(2,1) node {$\bullet$} -- (3,0) node {$\bullet$}  (4,0) node {$\bullet$}  (8,0.5) node {$[4],[3,4],[2,3],[2],[1]$};
\end{tikzpicture}\\
Notice that, in the given ordering, it is not an exceptional sequence, however if we shift module $[1]$ to the left of $[2]$, we get many braid equivalent exceptional collections. We study some properties of these in propositions \ref{typesVY}, \ref{numberofbodies}, and \ref{numberofbodies2}.
\item 
\begin{tikzpicture}
\draw (0,0) node {$\bullet$}--(1,1) node {$\bullet$}--(2,1) node {$\bullet$} -- (3,1) node {$\bullet$} --(4,0) node {$\bullet$}  (8,0.5) node {$[4],[3,4],[2,3],[1,2],[1]$};
\end{tikzpicture}
\item 
\begin{tikzpicture}
\draw (0,1) node {$\bullet$}--(1,0) node {$\bullet$}--(2,1) node {$\bullet$} -- (3,1) node {$\bullet$} --(4,0) node {$\bullet$}  (8,0.5) node {$[3,4],[3],[2,3],[1,2],[1]$};
\end{tikzpicture}
\end{enumerate}

To use induction for the enumerative problem, we introduce a gluing method. We describe two fundamental bones $Y$ and $V$ types as:
\begin{center}
\begin{tikzpicture}
\draw (0,1) node {$\bullet$}--(1,0) node {$\bullet$}--(2,1) node {$\bullet$}  (5,0) node {$\bullet$} --(6,1) node {$\bullet$} --(7,1 ) node {$\bullet$};
\draw (1,-1) node {V-bone\label{V}} (6,-1) node {Y-bone\label{Y}};
\end{tikzpicture}
\end{center}

They are in one to one correspondence with the exceptional collections:
\begin{enumerate}[label=\roman*)]
\item V-bone presents a projective module $P$, its subquotient $\faktor{P}{rad P}$ and one step shift $\sigma(P)$, i.e. $\left(P,\,\,\faktor{P}{rad P},\,\,\sigma(P)\right)$
\item Y-bone presents a simple module $S$, its injective envelope $I(S)$ and one step shift of it $\sigma(I(S))$ i.e. $\left(S,I(S),\sigma(S)\right)$.
\end{enumerate}

Here we describe explicitly all full exceptional sequences for $n=6$ by using the case $n=4$ and joint method for bones.

Y-bone can be joined to any full exceptional collection of $n=4$:

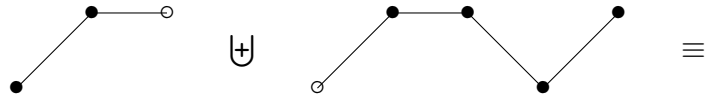
\begin{center}
\begin{tikzpicture}
\draw (0,0) node {$\bullet$}--(1,1) node {$\bullet$}--(2,1) node {$\circ$} (3,0.5) node {$\biguplus$}  (4,0) node {$\circ$} --(5,1) node {$\bullet$} -- (6,1) node {$\bullet$} --(7,0)node {$\bullet$}--(8,1)node {$\bullet$} (9,0.5) node {$\equiv$};
\end{tikzpicture}
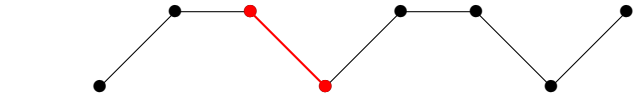
\vspace{1cm}
\begin{tikzpicture}
\draw (0,0) node {$\bullet$}--(1,1) node {$\bullet$}--(2,1) node {$\circ$}   (3,0) node {$\circ$} --(4,1) node {$\bullet$} -- (5,1) node {$\bullet$} --(6,0)node {$\bullet$}--(7,1)node {$\bullet$};
\draw[red,thick] (2,1 ) node {\textcolor{red}{$\bullet$}}--(3,0) node {\textcolor{red}{$\bullet$}} ;

\end{tikzpicture}\\
Collection is $\left([6],[5,6],[4,5],[4],[3,4],[2,3],[2],[1,2]\right)$
\end{center}

Other cases:

\begin{center}
\begin{tikzpicture}
\draw (0,0) node {$\bullet$}--(1,1) node {$\bullet$}--(2,1) node {$\circ$} (3,0.5) node {$\biguplus$}  (4,0) node {$\circ$} --(5,1) node {$\bullet$} -- (6,1) node {$\bullet$} --(7,0)node {$\bullet$} (8,0)node {$\bullet$} (9,0.5) node {$\equiv$};
\end{tikzpicture}\vspace{1cm}
\begin{tikzpicture}
\draw (0,0) node {$\bullet$}--(1,1) node {$\bullet$}--(2,1) node {$\circ$}   (3,0) node {$\circ$} --(4,1) node {$\bullet$} -- (5,1) node {$\bullet$} --(6,0)node {$\bullet$} (7,0)node {$\bullet$};
\draw[red,thick] (2,1 ) node {\textcolor{red}{$\bullet$}}--(3,0) node {\textcolor{red}{$\bullet$}} ;

\end{tikzpicture}\\
Collection is $\left([6],[5,6],[4,5],[4],[3,4],[2,3],[2],[1]\right)$
\end{center}

\vspace{2cm}

\begin{center}
\begin{tikzpicture}
\draw (0,0) node {$\bullet$}--(1,1) node {$\bullet$}--(2,1) node {$\circ$} (3,0.5) node {$\biguplus$}  (4,0) node {$\circ$} --(5,1) node {$\bullet$} -- (6,1) node {$\bullet$} --(7,1)node {$\bullet$}--(8,0)node {$\bullet$} (9,0.5) node {$\equiv$};
\end{tikzpicture}\vspace{1cm}
\begin{tikzpicture}
\draw (0,0) node {$\bullet$}--(1,1) node {$\bullet$}--(2,1) node {$\circ$}   (3,0) node {$\circ$} --(4,1) node {$\bullet$} -- (5,1) node {$\bullet$} --(6,1)node {$\bullet$}--(7,0)node {$\bullet$};
\draw[red,thick] (2,1 ) node {\textcolor{red}{$\bullet$}}--(3,0) node {\textcolor{red}{$\bullet$}} ;

\end{tikzpicture}\\
Collection is $\left([6],[5,6],[4,5],[4],[3,4],[2,3],[1,2],[1]\right)$
\end{center}

\begin{center}
\begin{tikzpicture}
\draw (0,0) node {$\bullet$}--(1,1) node {$\bullet$}--(2,1) node {$\circ$} (3,0.5) node {$\biguplus$}  (4,1) node {$\circ$} --(5,0) node {$\bullet$} -- (6,1) node {$\bullet$} --(7,1)node {$\bullet$}--(8,0)node {$\bullet$} (9,0.5) node {$\equiv$};
\end{tikzpicture}\vspace{1cm}
\begin{tikzpicture}
\draw (0,0) node {$\bullet$}--(1,1) node {$\bullet$}--(2,1) node {$\circ$}   (3,1) node {$\circ$} --(4,0) node {$\bullet$} -- (5,1) node {$\bullet$} --(6,1)node {$\bullet$}--(7,0)node {$\bullet$};
\draw[red,thick] (2,1 ) node {\textcolor{red}{$\bullet$}}--(3,1) node {\textcolor{red}{$\bullet$}} ;

\end{tikzpicture}\\
Collection is $\left([6],[5,6],[4,5],[3,4],[3],[2,3],[1,2],[1]\right)$
\end{center}

There is one more full exceptional sequence which starts by $[6]$. We can put a discrete dot in front of module $[4]$ since $[1]$ and $[4]$ cannot extend each other in $\Lambda^2_6$. Therefore we need to consider the following joint operation:

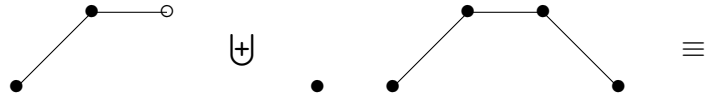
\begin{center}
\begin{tikzpicture}
\draw (0,0) node {$\bullet$}--(1,1) node {$\bullet$}--(2,1) node {$\circ$} (3,0.5) node {$\biguplus$}  (4,0) node {$\bullet$} (5,0) node {$\bullet$} -- (6,1) node {$\bullet$} --(7,1)node {$\bullet$}--(8,0)node {$\bullet$} (9,0.5) node {$\equiv$};
\end{tikzpicture}
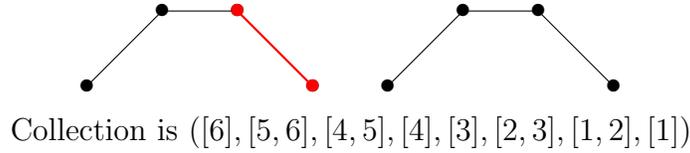
\vspace{1cm}
\begin{tikzpicture}
\draw (0,0) node {$\bullet$}--(1,1) node {$\bullet$}--(2,1) node {$\circ$}  -- (3,0) node {$\circ$} (4,0) node {$\bullet$} -- (5,1) node {$\bullet$} --(6,1)node {$\bullet$}--(7,0)node {$\bullet$};
\draw[red,thick] (2,1 ) node {\textcolor{red}{$\bullet$}}--(3,0) node {\textcolor{red}{$\bullet$}} ;

\end{tikzpicture}\\
Collection is $\left([6],[5,6],[4,5],[4],[3],[2,3],[1,2],[1]\right)$
\end{center}

This gives all possible bones starting by simple module $[6]$. This is equivalent to all possible joints by $Y$ bone.
There is only one possible joint for $V$ bone :
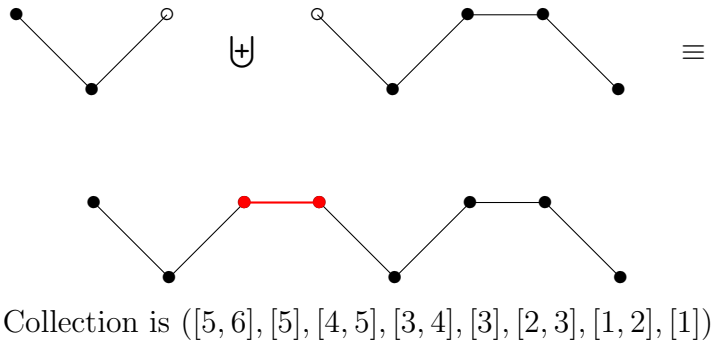
\begin{center}
\begin{tikzpicture}
\draw (0,1) node {$\bullet$}--(1,0) node {$\bullet$}--(2,1) node {$\circ$} (3,0.5) node {$\biguplus$}  (4,1) node {$\circ$} --(5,0) node {$\bullet$} -- (6,1) node {$\bullet$} --(7,1)node {$\bullet$}--(8,0)node {$\bullet$} (9,0.5) node {$\equiv$};
\end{tikzpicture}\vspace{1cm}
\begin{tikzpicture}
\draw (0,1) node {$\bullet$}--(1,0) node {$\bullet$}--(2,1) node {$\circ$}   (3,1) node {$\circ$} --(4,0) node {$\bullet$} -- (5,1) node {$\bullet$} --(6,1)node {$\bullet$}--(7,0)node {$\bullet$};
\draw[red,thick] (2,1 ) node {\textcolor{red}{$\bullet$}}--(3,1) node {\textcolor{red}{$\bullet$}} ;

\end{tikzpicture}\\
Collection is $\left([5,6],[5],[4,5],[3,4],[3],[2,3],[1,2],[1]\right)$
\end{center}

We cannot add $V$ type bone to a bone which starts by $Y$ type bone, so there is only one case.

\begin{proposition}\label{typesVY}
If $n$ is even, there are $n-1$ distinct bones which present full weak exceptional collections starting by module $[n]$ i.e. $Y$ type, and there is one bone which stands for exceptional collection starting by module $[n-1,n]$ i.e. $V$ type.
\end{proposition}

\begin{proof}
Statement holds for $n=4,6$ by direct computation as we discussed. Assume that this holds for $n$. To reach $n+2$ case, all bones can be joined by $Y$ type, which gives $n$ distinct types. There is only one extension by $V$ bone. There is one remaining type, which is another extension by $Y$ type. Notice that in $\Lambda^2_{n+2}$, modules $[n]$ and $[1]$ cannot extend each other as analogous to the case \ref{nisodd}. Therefore, in total we get $n+2$ distinct bones.
\end{proof}

To count the number of exceptional collections, we need to know the structure of all possible bones. It turns out that there can be at most two main bodies.

\begin{proposition}\label{numberofbodies} A bone of $\Lambda^2_n$ cannot have more than two main bodies.
\end{proposition}
\begin{proof}
We prove it by induction. We identified all possible weak exceptional sequences by bones and construction is iterative i.e. extension of weak exceptional sequences of $\Lambda^2_{n-2}$ by $Y$ bone and $V$ bone cannot increase the number of main bodies, by the induction hypothesis, it has to be at most two.
\end{proof}

Now by using star-bar combinatorics we can count all possible exceptional sequences indexed by bodies. If there is only one main body, braid group acts trivially. Therefore we need to understand two body cases.

\begin{proposition}\label{braid1}
If there are two main bodies, the number of exceptional collections is given by:
\begin{align}\label{c}
c(x,y)=\sum\limits^{x}_{j=2}\binom{y+j-2}{j-1}=\binom{y+x-1}{y}-1
\end{align}
where $x$ and $y$ count the number of black dots i.e modules in each main body.
\end{proposition}

\begin{proof}
Assume that weak exceptional sequences of the first and the second bodies are $E=(E_1,\ldots,E_x)$ and $F=(F_1,\ldots,F_y)$. To use transitive action of $\sigma$, we fix the position of $E_1$, it is the first module in the exceptional sequence obtained by the unification of $E$ and $F$. The unified sequence is exceptional if the following conditions satisfied:
\begin{itemize}
\item The order of modules coming from the same body should be preserved.
\item Position of the last module $E_x$ should be at the right side of the first module $F_1$, otherwise there is nontrivial extension between them.
\end{itemize}

Now, we can list possible cases and then count them. Assume that $F_1$ is between $E_{x-1}$ and $E_x$. To replace remaining $F_i$s there are two blocks and the total number is $y$ by stars-bars combinatorics. Similarly, if $F_1$ is between $E_{x-2}$ and $E_{x-1}$, there are three blocks to put $y-1$ elements, hence $\binom{y+1}{2}$. In general, if $F_1$ is between $E_{x-i+1}$ and $E_{x-i}$ then there are $i+1$ blocks to distribute $y-1$ elements which gives $\binom{y+i-1}{i}$.
Therefore the sum $c(x,y)=\sum\limits^{x}_{j=2}\binom{y+j-2}{j-1}$ gives the total number of weak exceptional sequences associated to bone of two main bodies of sizes $x$, $y$. And simple computation involving sums of binomial coefficients verifies the last equality.
\end{proof}

\begin{lemma}\label{numberofbodies2} The number of bones which has one main bodies is $\cfrac{n}{2}+1$.  
The number of bones which has two main bodies is $\cfrac{n}{2}-1$. 
\end{lemma}
\begin{proof}
We use induction. To get two bodies in $\Lambda^2_{n+2}$, $Y$ bone extensions should be applies to two bodied bodies of $\Lambda^2_n$. This gives $\frac{n}{2}-1+1=\frac{n}{2}$ by proposition \ref{typesVY}. 
\end{proof}

\begin{proposition}\label{braid2}
The  number of full weak exceptional collections which braid group acts nontrivially is given by the sum
\begin{align*}
\sum\limits^{\frac{3}{2}n-2}_{x=4,\, (x)_3\equiv 1}c(x,\frac{3}{2}n-x-1)
\end{align*}
\end{proposition}

\begin{proof}
For simplicity, we substitute $n=2k$, and use \ref{c} to get:
\begin{align*}
\sum\limits^{\frac{3}{2}n-2}_{x=4,\, (x)_3\equiv 1}c(x,\frac{3}{2}n-x-1)=1-k+\sum\limits^{k-1}_{j=1}\binom{3k-2}{3j}
\end{align*}

and consider the sum over $j$:
\begin{align*}
\sum\limits^{k-1}_{j=1}c(3j+1,3k-3j-2)
\end{align*}
In proposition \ref{braid1}, we obtained a counting function for weak exceptional sequences associated to a bone. Here we consider the summation over all possible bones with two main bodies. Each connected body has $3j+1$ modules and the total is $3k-1$, hence the formula follows. 
\end{proof}

\begin{lemma}\label{eq1} We have the following equality:
\begin{align*}
\sum\limits^n_{j\equiv 0\!\!\! \mod 3} \binom{n}{j}=\cfrac{1}{3}\left(2^n+2\cos\left(\frac{n\pi}{3}\right)\right)
\end{align*}
\end{lemma}
\begin{proof}
Let $\omega=e^{\frac{2i\pi}{3}}$ be the third root of unity. Since it satisfies \begin{align*}
 1+\omega^j+\omega^{-j}=\begin{cases}
3& \text{if}\quad j\equiv 0\mod 3\\
0& \text{if}\quad j\not\equiv 0\mod 3
\end{cases}
\end{align*}
we have:
\begin{align*}
\sum\limits^n_{j\equiv 0\!\!\! \mod 3} \binom{n}{j}&=\sum\limits^n_{j=0} \binom{n}{j} \cfrac{1}{3}\left(1+\omega^j+\omega^{-j}\right) \\
&=\cfrac{1}{3}\left(2^n+(1+\omega)^n+(1+\omega^{-1})^n\right)\\
&=\cfrac{1}{3}\left(2^n+\omega^{n/2}+\omega^{-n/2}\right)\\
&=\cfrac{1}{3}\left(2^n+2\cos\left(\frac{n\pi}{3}\right)\right)
\end{align*}

\end{proof}

\begin{theorem}[Thm \ref{thm3} b]\label{niseven}
If $n=2k$, the total number of full weak exceptional sequences is:
\begin{align}
\#(\Lambda^2_{2k})=2k\left(1+\cfrac{8^k}{12}-\cfrac{(-1)^k}{3}\right)
\end{align}

\end{theorem}

\begin{proof}

It is enough to consider exceptional collections starting with the modules $[n]$ or $[n-1,n]$. The previous result \ref{braid2} gives the number of exceptional sequences with braid group action restricted to each bone. The remaining ones, i.e. module interchange is not allowed, are $k+1$ by the lemma \ref{numberofbodies2}. Therefore we get the sum :
\begin{align*}
1+k+\sum\limits^{k-1}_{j=1}c(3j+1,3k-3j-2)
\end{align*}
Since period of $\sigma$ is $n=2k$, and the action is transitive:
\begin{align*}
\#(\Lambda^2_{2k})=2k\left(1+k+\sum\limits^{k-1}_{j=1}c(3j+1,3k-3j-2)\right)
\end{align*}
By proposition \ref{c} and lemma \ref{eq1} we can simplify it and get: 
\begin{align*}
\#(\Lambda^2_{2k})=2k\left(1+\cfrac{8^k}{12}-\cfrac{(-1)^k}{3}\right)
\end{align*}
\end{proof}

\section{Standard vs. Weak Exceptional Sequences}\label{last}

We compare some properties of exceptional sequences and weak exceptional sequences. Let $L\Lambda^2_n$ be linear radical square zero Nakayama algebra of rank $n$. Then the number of exceptional sequences is equal to the number of idempotent functions \cite{sen19} which is $\sum\limits^n_{j=1}\binom{n}{j}j^{n-j}$. By simple computation we have:
\begin{align*}
\lim\limits_{n\rightarrow\infty} \cfrac{\#(\Lambda^2_{n})}{\#(L\Lambda^2_n)}=0 &\quad \text{and}\quad\lim\limits_{n\rightarrow\infty} \cfrac{\#(\mathbb{A}_n)}{\#(\Lambda^n_n)}=0
\end{align*}

The first limit is zero, since in the denominator, there is a summand of growth rate $k^k$. Moreover the full collection of $\Lambda^2_n$ is of size $\frac{3}{2}n-1$ which is greater than $n$ for large $n$. We do not know a closed formula for weak exceptional sequences of $\Lambda^2_n$ of fixed size smaller than the full size. \\

Recall that an exceptional sequence $E=(M_1,\ldots,M_k)$ is called orthogonal \cite{ringel1994braid} if $\Hom(M_i,M_j)=0$ for $i\neq j$ . It can be verified that the full weak exceptional sequences of type $\Lambda^2_n$ are not orthogonal, which is different than the exceptional sequences of radical square zero Nakayama algebras \cite{sen19}. Because there is an exceptional sequence of simple modules in $L\Lambda^2_n$ which is orthogonal collection. 

The following proposition and its corollary explains why we consider weak exceptional sequences of selfinjective Nakayama algebras:

\begin{proposition}\label{weakvs} Let $\Lambda_n$ be a cyclic connected Nakayama algebra of rank $n$. An indecomposable  $\Lambda$-module is exceptional in the sense of \ref{standar} if and only if it is not periodic module and the length is smaller than or equal to $n$.
\end{proposition}
\begin{proof}
If a module is periodic, then $X\cong\Omega^j(X)$ for some $j$, hence there is nontrivial higher extension $\ext^{j}(X,X)$. And $\dim\Hom(X,X)\geq 2$ if and only if $top(X)$ appears in the composition series of radical of $X$. 
\end{proof}

\begin{corollary} A selfinjective Nakayama algebra cannot have a standard nonprojective exceptional module.
\end{corollary}
\begin{proof}
A Nakayama algebra is self injective if and only if every nonprojective module is periodic \cite{fidim}. Therefore, these modules cannot be exceptional by \ref{weakvs}. 
\end{proof}

Now, we give an example of the full weak collection whose subcollections are standard. Let $\Lambda$ be Nakayama algebra given by the relations $\alpha_3\alpha_2\alpha_1=0$ and $\alpha_1\alpha_3=0$. Since $\bm\varepsilon(\Lambda)$ is self injective \cite{syz}, every indecomposable $\Lambda$-module except 
$\begin{vmatrix} 1\end{vmatrix}, \begin{vmatrix}
2\\3\end{vmatrix}$ (since they correspond to simple mod-$\bm\varepsilon(\Lambda)$) are exceptional by proposition \ref{weakvs}. In particular, this implies that there is no orthogonal complete exceptional sequence. The sequence \begin{align}
\left(\begin{vmatrix}2\end{vmatrix},\begin{vmatrix} 1\\2\end{vmatrix},\begin{vmatrix} 3\\1\end{vmatrix},\begin{vmatrix} 3\end{vmatrix}\right)
\end{align}

in the bone structure \ref{bone} is weak exceptional, and the subsequences: $\left(\begin{vmatrix}2\end{vmatrix},\begin{vmatrix} 1\\2\end{vmatrix},\begin{vmatrix} 3\\1\end{vmatrix}\right)$, $\left(\begin{vmatrix} 1\\2\end{vmatrix},\begin{vmatrix} 3\\1\end{vmatrix},\begin{vmatrix} 3\end{vmatrix}\right)$ are exceptional, since $\ext^2_{\Lambda}(\vert 3\vert,\vert 2\vert)$ is nontrivial.

\bibliographystyle{alpha}
\newcommand{\etalchar}[1]{$^{#1}$}

\end{document}